\documentclass[reqno,12pt]{amsart}
\usepackage{hyperref}
\usepackage{amsfonts}
\usepackage{amssymb}
\usepackage{graphicx}
\usepackage{pstricks}
\usepackage{amsmath}
\usepackage{amsxtra}
\usepackage{url}

\usepackage{cite}
\usepackage{mathrsfs, amsthm, xifthen, verbatim, mathrsfs}
\usepackage{ulem}

\usepackage{moreverb}

\usepackage{amsmath,amssymb,amsfonts,epsfig,color}

\usepackage{multicol}

\definecolor{DarkBlue}{rgb}{0,0.2,0.6}
\definecolor{PinkPurple}{rgb}{0.8,0.3,0.3}
\definecolor{darkgreen}{rgb}{.1,.5,0}
\definecolor{brown}{rgb}{.4,.2,.1}

\newtheorem{thm}{Theorem}[section]
\newtheorem{prop}[thm]{Proposition \!\!}
\newtheorem{cor}[thm]{Corollary \!\!}
\newtheorem{lem}[thm]{Lemma \!\!}
\newtheorem{ex}[thm]{Example \!\!}
\newtheorem{remark}[thm]{Remark \!\!}

\newtheorem{definition}[thm]{Definition}
\newtheorem{question}[thm]{Question \!\!}
\def\ker{\operatorname{ker}}

\newcommand\ddfrac[2]{\frac{\displaystyle #1}{\displaystyle #2}}

\numberwithin{equation}{section}

\begin{document}

\title{Moment Infinite Divisibility of Weighted Shifts: Sequence Conditions}
\author{Chafiq Benhida}
\address{UFR de Math\'{e}matiques, Universit\'{e} des Sciences et
Technologies de Lille, F-59655 \newline Villeneuve-d'Ascq Cedex, France}
\email{chafiq.benhida@univ-lille.fr}

\author{Ra\'ul E. Curto}
\address{Department of Mathematics, University of Iowa, Iowa City, Iowa 52242-1419, USA}
\email{raul-curto@uiowa.edu}

\author{George R. Exner}
\address{Department of Mathematics, Bucknell University, Lewisburg, Pennsylvania 17837, USA}
\email{exner@bucknell.edu}

\subjclass[2010]{Primary 47B20, 47B37, Secondary 44A60.}

\keywords{Weighted shift, Subnormal, Moment infinitely divisible, Completely monotone sequence, Completely alternating sequence, Aluthge transform}

\begin{abstract}
We consider weighted shift operators having the property of moment infinite divisibility; that is, for any $p > 0$, the shift is subnormal when every weight (equivalently, every moment) is raised to the $p$-th power. \  By reconsidering sequence conditions for the weights or moments of the shift, we obtain a new characterization for such shifts, and we prove that such shifts are, under mild conditions, robust under a variety of operations and also rigid in certain senses. \  In particular, a weighted shift whose weight sequence has a limit is moment infinitely divisible if and only if its Aluthge transform is. \  We also consider back-step extensions, subshifts, and completions.

\end{abstract}

\maketitle

\tableofcontents

\setcounter{tocdepth}{2}

%%%%%%%%%%%%%%%%%%%%%%%%
%%%%%%%%%%%%%%%%%%%%%%%%
%%%%%%%%%%%%%%%%%%%%%%%%
%%%%%%%%%%%%%%%%%%%%%%%%
%%%%%%%%%%%%%%%%%%%%%%%%
%%%%%%%%%%%%%%%%%%%%%%%%

\section{Introduction and Preliminaries} \label{Intro}

Let $\mathcal{H}$ denote a separable, complex Hilbert space and $\mathcal{L}(\mathcal{H})$ be the algebra of bounded linear operators on $\mathcal{H}$. \  Recall that an operator $T$ is subnormal if it is the restriction to a (closed) invariant subspace of a normal operator, and hyponormal if $T^* T \geq T T^*$. \  We are concerned in this paper with (unilateral) weighted shift operators $W_\alpha$ on the classical sequence space $\ell^2$, with weight sequence $\alpha: \alpha_0, \alpha_1, \ldots$, which are not only subnormal, but which have the stronger property that, for each $p \geq 0$, the shift $W_{\alpha}^{(p)}$ corresponding to the weight sequence $\alpha^{(p)}: \alpha_0^p, \alpha_1^p, \ldots$ remains subnormal. \  Such shifts are called \textit{moment infinitely divisible} (see \cite{BCE} for an initial study and justification of this terminology), and we write $\mathcal{MID}$ for the class of such shifts. \  

In this paper we consider approaches to this problem based on properties of the sequence of weights or the sequence of moments (definitions reviewed below). \  These approaches fall in the Agler-Embry or ``$n$-contractivity'' approach to subnormality of contraction operators; in a subsequent paper we will consider approaches rooted in the Bram-Halmos ``$k$-hyponormality'' approach to subnormality.

The organization of this paper is as follows. \  In Section \ref{mainresults} we present the main results of this paper, and their relevance in and connections with the existing literature. \ In the remainder of this section we shall give notation and terminology and certain needed or motivational background results. \  In Section \ref{Sect2} we will first modify a result from \cite{BCE} to give a new characterization of moment infinite divisibility. \ In Section \ref{Sect3} we will present some consequences of the tools in Section \ref{Sect2}, including in particular discussion of the Aluthge transform and the result that under mild hypotheses a weighted shift is $\mathcal{MID}$ if and only if its Aluthge transform is.

To set the notation for weighted shifts, let $\mathbb{N}_0 = \{0, 1, \ldots\}$ and denote by $\ell^2$ the classical Hilbert space $\ell^2(\mathbb{N}_0)$ with canonical basis $e_0, e_1, \ldots$ (note that we begin indexing at zero). \  Let $\alpha: \alpha_0, \alpha_1, \ldots$ be a (bounded) weight sequence and let $W_\alpha$ be the weighted shift defined by its action on the basis: $W_\alpha e_j = \alpha_j e_{j+1}$ and the assumption of linearity. \  For all questions of interest in this paper, it is without loss of generality to assume that the weight sequence $\alpha$ is positive, and we do henceforth without further mention. \ The moments $\gamma = (\gamma_n)_{n=0}^\infty$ of the shift are defined by $\gamma_0 = 1$ and $\gamma_n = \prod_{j=0}^{n-1} \alpha_j^2$ for $n \geq 1$. \ We recall (see \cite[III.8.16]{Con},\cite{GW}) that a weighted shift $W_\alpha$ is subnormal if and only if it has a Berger measure, meaning a probability measure $\mu$ supported on $[0, \|W_\alpha\|^2]$ and satisfying
$$
\gamma_n = \int_0^{\|W_\alpha\|^2} t^n d \mu(t), \hspace{.2in} n = 0, 1, \ldots .
$$

Applying the Cauchy-Schwarz inequality in $L^2(\mu)$ to the monomials $s^{n/2}$ and $s^{(n+2)/2}$, one obtains $\gamma_{n+1}^2 \le \gamma_n \gamma_{n+2} \; (n \ge 0)$, and therefore $\alpha_n^2 \le \alpha_{n+1}^2 \; (n \ge 0)$. \ Recall that an operator $T$ satisfying $T^*T-TT^* \ge 0$ is called hyponormal, and it is readily seen that for a weighted shift $W_{\alpha}$ this is exactly $\alpha_n^2 \le \alpha_{n+1}^2$ for all $n \ge 0$.  

The Agler-Embry characterization of subnormality for a contractive operator $T$ (that is, $\|T \| \leq 1$) is based on the notion of $n$-contractivity (see \cite{Ag}). \  An operator is $n$-contractive, $n = 1, 2, \ldots$, if
\begin{equation} \label{eq11}
A_n(T):=\sum_{i=0}^n (-1)^i \binom{n}{i} {T^*}^i T^i \geq 0.
\end{equation}
The characterization is that a contractive operator is subnormal if and only if it is $n$-contractive for all positive integers $n$. \ It is well known, and simply a computation following from the fact that ${W_\alpha^*}^i W_\alpha^i$ is diagonal, that for a shift it suffices to test this condition on basis vectors and that a weighed shift is $n$-contractive if and only if
$$
\sum_{i=0}^n (-1)^i \binom{n}{i} \gamma_{k+i} \geq 0, \qquad k = 0, 1, \ldots.
$$
Let $\nabla$ be the forward difference operator on sequences, that is, for any sequence $a = (a_j)_{j=0}^\infty$,  
$$
(\nabla a)_j := a_j - a_{j+1},
$$
and define the iterated forward difference operators $\nabla^{n}$ by 
$$
\nabla^{0} a := a \; \; \textrm{ and } \; \; \nabla^{n} := \nabla (\nabla^{n-1}),
$$
for $n \geq 1$. \ To ease the notation slightly in some settings, set, for any $n \geq 1$ and $k \geq 0$,
$$
T_a(n,k) := (\nabla^{n} a)_k = \sum_{i=0}^n (-1)^i \binom{n}{i} a_{i+k}
$$
(where we will suppress the dependence on the sequence $a$ and write simply $T(n,k)$ if no confusion will arise). \ With a slight abuse of language say that a sequence $a$ is $n$-contractive if $T_a(n,k) \geq 0$ for all $k = 0, 1, \ldots$. \  We will similarly employ
$$
LT_a(n,k) := (\nabla^{n} \ln a)_k = \sum_{i=0}^n (-1)^i \binom{n}{i} \ln a_{i+k}
$$
when we wish to consider properties arising from the logs of the sequence terms. \ 
In the case of a weighted shift $W_{\alpha}$, it is possible to establish, via mathematical induction on $n$, that  
\begin{equation} \label{eq1}
LT_{\gamma}(n+1,k)=-LT_{\alpha^2}(n,k)=-2LT_{\alpha}(n,k),
\end{equation}
where $\gamma$ denotes the sequence of moments, and $\alpha^2$ the sequence of weights squared. \ (The induction step requires the binomial identity $\binom{n+2}{i}=\binom{n+1}{i}+\binom{n+1}{i-1}$ and the fact that $\ln \gamma_{i+k+1}=\ln \gamma_{i+k} + \ln \alpha_{i+k}^2$.)

There is alternative language for these and related concepts:  a sequence $a$ is $n$--monotone if $T_a(n,k) = (\nabla^{n} a)_k \geq 0$ for all $k = 0, 1, \ldots$, $n$--hypermonotone if it is $j$--monotone for all $j = 1, \ldots, n$, and completely monotone if it is $n$--monotone for all $n = 1, 2, \ldots$. \  A sequence is $n$--alternating if $T_a(n,k) = (\nabla^{n} a)_k \leq 0$ for all $k = 0, 1, \ldots$, $n$--hyperalternating if it is $j$--alternating for all $j = 1, \ldots, n$, and completely alternating if it is $n$--alternating for all $n = 1, 2, \ldots$. \  We shall say that a sequence is $n$--log monotone (respectively, completely log monotone, $n$--log alternating, completely log alternating) if the sequence $(\ln a_j)$ is $n$--monotone (respectively, completely monotone, $n$-alternating, completely alternating).

Recall that an operator $T$ has a polar decomposition $T = U |T|$, where $|T| = (T^* T)^{1/2}$ and $U$ is a partial isometry satisfying the kernel condition $\ker U = \ker |T| = \ker T$. \  The much-studied Aluthge transform of $T$ (see, for example, \cite{Alu,APS,DySc,JKP1,Ya}) is $AT(T) := |T|^{1/2} U |T|^{1/2}$, and we define the iterated Aluthge transform $AT^{(n)}$ in the obvious way. \  It is merely a computation that if $W_\alpha$ is a weighted shift then $AT(W_\alpha)$ is again a weighted shift with weight sequence
$$
\sqrt{\alpha_0 \alpha_1}, \sqrt{\alpha_1 \alpha_2}, \ldots.
$$

It is useful to note, in considering $\mathcal{MID}$ weighted shifts, that raising every weight to the $p$-th power is equivalent to raising every moment to the $p$-th power. \ We refer the reader to \cite{BCE} for an initial study of moment infinite divisibility;  this paper constitutes a continuation of that study.

We will reference the $k$-hyponormality approach to subnormality only in passing, but for the convenience of the reader we give a very abbreviated version of this background (see \cite{Cu1} for a full discussion). \ Recall that it is the Bram-Halmos characterization of subnormality (see \cite{Br}) that an operator $T$ is subnormal if and only if, for every $k = 1, 2, \ldots$, a certain $(k+1) \times (k+1)$ operator matrix $A_n(T)$ is positive. \ Given $k \ge 1$, an operator is $k$-hyponormal whenever this positivity condition holds for $k$. \ For weighted shifts, it is well known that $k$-hyponormality reduces to the positivity, for each $n$, of the $(k+1) \times (k+1)$ Hankel moment matrix $A(n,k)$ \cite[Theorem 4]{Cu}, where
$$A(n,k) = \left(
\begin{array}{cccc}
\gamma _{n} & \gamma _{n+1} & \cdots & \gamma _{n+k} \\
\gamma _{n+1} & \gamma _{n+2} & \cdots & \gamma _{n+k+1} \\
\vdots & \vdots & \ddots & \vdots \\
\gamma _{n+k} & \gamma _{n+k+1} & \cdots & \gamma _{n+2k}%
\end{array}
\right). $$
The Bram-Halmos/$k$-hyponormality approach to subnormality has been more thoroughly studied than the Agler-Embry/$n$-contractivity approach; see \cite{Cu1} and \cite{CF1} for some of the beginnings of this considerable body of work.

%%%%%%%%%%%%%%%%%%%%%%%%
%%%%%%%%%%%%%%%%%%%%%%%%
%%%%%%%%%%%%%%%%%%%%%%%%
%%%%%%%%%%%%%%%%%%%%%%%%
%%%%%%%%%%%%%%%%%%%%%%%%
%%%%%%%%%%%%%%%%%%%%%%%%

\section{Description of Main Results} \label{mainresults}

In this section we list the main results of this paper, and indicate how they connect with results in the existing literature. \ First, recall that a sequence $\gamma$ of {\it moments} is log completely monotone if $LT_{\gamma}(n,k) \equiv \sum_{i=0}^n (-1)^i \binom{n}{i} \ln \gamma_{i+k} \ge 0$ for all $n \ge 1$ and $k \ge 0$. \ On the other hand, a sequence $\alpha^2$ of {\it weights squared} is completely alternating if $T_{\alpha^2}(n,k) \equiv \sum_{i=0}^n (-1)^i \binom{n}{i} \alpha_{i+k}^2 \le 0$ for all $n \ge 1$ and $k \ge 0$. \ It is well known that a general sequence $\psi$ is completely alternating if and only if the sequence $\varphi_t:=e^{-t \psi}$ is completely monotone for all $t>0$ \cite[Prop. 6.10]{BCR}. \ In \cite[Theorem 3.1]{BCE} we proved that a contractive weighted shift $W_{\alpha}$ is $\mathcal{MID}$ if and only if the sequence of weights squared $\{\alpha_n^2\}$ is log completely alternating. \ 

Our first main result is Proposition \ref{prop:wtslogCAvsmomlogCM} (and its companion Theorems \ref{thm24} and \ref{thm25}), asserting that 
$$
\alpha^2 \textrm{ is log completely alternating} \Longleftrightarrow \gamma^{(\alpha)} \textrm{ is log completely monotone},
$$
where $\gamma^{(\alpha)}$ is the moment sequence of $\alpha$. \ Thus, we establish a two-way bridge between weights squared sequences and the associated moment sequences, not known before. \ This is represented by the double-arrow in the diagram on page \pageref{diagram}.

Next, we substantially expand the collection of sequences which are known to be log completely alternating by proving, in Theorem \ref{cor:simplyweightquotient}, that a non-decreasing sequence $a$ is log completely alternating if and only if the quotient sequence $\{ b_n:=\frac{a_n}{a_{n+1}} \}$ is log completely alternating. \ With these results in hand, we proceed to establish a link with the Aluthge transform: given a contractive weighted shift $W_{\alpha}$ such that $\lim_{n \rightarrow \infty} \alpha_n$ exists (so, in particular, if $W_{\alpha}$ is hyponormal), then  
$$
AT(W_{\alpha}) \in \mathcal{MID} \Leftrightarrow W_{\alpha} \in \mathcal{MID}.
$$
This provides a new way to determine whether a hyponormal contractive weighted shift is $\mathcal{MID}$. \ In Corollary \ref{cor45} we prove that, in fact, any of the iterated Aluthge transforms of $W_{\alpha}$ can be used instead of $AT(W_{\alpha})$. \ We also complete a full circle of equivalent conditions, by establishing that for contractive hyponormal weighted shifts, 
$$
W_{\alpha}^{(1/2)} \in \mathcal{MID} \Leftrightarrow W_{\alpha} \in \mathcal{MID} \Leftrightarrow AT(W_{\alpha}) \in \mathcal{MID},
$$
where $W_{\alpha}^{(1/2)}$ is the Schur square root of $W_{\alpha}$. \ We do this in Corollary \ref{cor47}, and bring closure to a question in \cite{CuEx}.

Our next main result is Theorem \ref{prop:ATbijective}, in which we prove that the Aluthge transform, when restricted to the class $\mathcal{MID}$, is a bijection. \ This means, in particular, that every Agler shift $A_k$ is the Aluthge transform of a contractive weighted shift. \ In Proposition \ref{prop417}, we calculate the precise form of $(AT)^{-1}(A_k)$. \ (Recall that the $k$-th Agler shift $A_k$ is that with weight sequence $\{\alpha(k)_n\}_{n=0}^\infty$ with $\alpha(k)_n = \sqrt{\frac{n+1}{n+k}}$ for $k = 2, 3, \ldots$ and $n = 0, 1, \ldots$.) 

In Proposition \ref{prop:LTeq0Flat}, we establish a flatness result for the class $\mathcal{MID}$. \ Concretely, we prove that if a contractive weighted shift $W_{\alpha}$ is $\mathcal{MID}$, and for some $n \ge 1$ and $k \ge 0$ one has $LT_{\alpha}(n,k)=0$, then $W_{\alpha}$ is flat; that is, $\alpha_0 \le \alpha_1 = \alpha_2 = \alpha_3 = \cdots.$ 

Along the way we present operations on sequences which involve quotients, Aluthge transforms, generalized Aluthge transforms, limits, and subshifts; our study includes several examples of $\mathcal{MID}$ unilateral weighted shifts, even some whose subnormality was not previously known.

Although we do not pursue this point of view here, each $\mathcal{MID}$ weighted shift yields, in the Hankel matrices of its moments, examples of infinitely divisible matrices (see \cite{Bha}), some of which are also new.

%%%%%%%%%%%%%%%%%%%%%%%%
%%%%%%%%%%%%%%%%%%%%%%%%
%%%%%%%%%%%%%%%%%%%%%%%%
%%%%%%%%%%%%%%%%%%%%%%%%
%%%%%%%%%%%%%%%%%%%%%%%%
%%%%%%%%%%%%%%%%%%%%%%%%

\section{Sequence Conditions for Moment Infinite Divisibility} \label{Sect2}

We turn shortly to proving some results comparing various properties for a real sequence (in the sequel, usually to be a sequence of weights or moments, or their logs, of a shift). \ We begin with one fundamental to the point of view of this paper, and deduce a new characterization of moment infinite divisibility for weighted shifts.  

Our starting point is a characterization of moment infinite divisibility for contractive weighted shifts in terms of sequences (cf. \cite{BCE}).

\begin{thm}\cite[Theorem 3.1]{BCE} \label{th:BCEmomlogID}
A contractive weighted shift $W_\alpha$ is moment infinitely divisible (in symbols, $W_{\alpha} \in \mathcal{MID}$) if and only if the sequence of weights squared (or, equivalently, the sequence of weights) is log completely alternating.
\end{thm}

We seek a relationship between log completely alternating weights and log completely monotone moments for a shift, but first need a lemma.

\begin{lem}
Let $n, k$ be positive integers with $0\leq k < n$. \  Then
\begin{equation}
\sum_{i=0}^k (-1)^i \binom{n}{i} = (-1)^k \binom{n-1}{k}.
\end{equation}
\end{lem}

\begin{proof}
For a fixed $n$, do finite induction on $k$, and use Pascal's identity.  
\end{proof}

\medskip

\begin{prop}   \label{prop:wtslogCAvsmomlogCM}
The weights squared $(w_n^2)_{n=0}^\infty$ (equivalently, the weights) of a contractive weighted shift are log completely alternating if and only if the moments of the shift are log completely monotone.
\end{prop}

\begin{proof}
Suppose the weights squared are log completely alternating. \ It is a computation to check from the definitions that $(\nabla \ln \gamma)_n = - (\ln w^2)_n$ for all $n \geq 0$. \ Since each $w_n \leq 1$, we have that $\gamma$ is $1$--log monotone. \ It is well known that for any sequence $a$ and any $k$, 
$$
(\nabla^{k+1} a)_n = (\nabla^k a)_n - (\nabla^k a)_{n+1} \; (\textrm{for all } n \ge 0),
$$
and this together with an induction based on $(\nabla \ln \gamma)_n = - (\ln w^2)_n$ yields (as in \cite[Remark 2.6]{BCE}) that 
$$
(\nabla^{k+1} \ln \gamma)_n=-(\nabla^k (\ln w^2))_n
$$
for all $k,n \geq 0$. \ Since $(w_n^2)$ is log completely alternating, it results that $\gamma$ is log completely monotone. \ As the steps above are reversible, and using again that $w_n \leq 1$ for all $n$, the reverse implication follows as well. 
\end{proof}

\medskip

The following characterization is then immediate.

\begin{thm} \label{thm24}
A contractive weighted shift $W_\alpha$ is $\mathcal{MID}$ if and only if the sequence of moments is log completely monotone.
\end{thm}

We digress for a moment to point out a fundamental difference between conditions on moments and conditions on weights. \  In general, conditions on weights are robust under scaling, and in particular scaling is harmless for subnormality or moment infinite divisibility results, as indicated by the following in comparison with/improvement of Theorem \ref{th:BCEmomlogID}.

\begin{thm} \label{thm25}
Suppose $W_\alpha$ is a weighted shift with weight sequence $\alpha$. \  If the weight sequence $\alpha$ is log completely alternating (equivalently, the squared sequence $(\alpha^2_n)_{n=0}^\infty$ is log completely alternating), then $W_\alpha$ is $\mathcal{MID}$ (and therefore subnormal).
\end{thm}

\begin{proof}
Supposing first that $W_\alpha$ is a contraction, this follows immediately from Theorem \ref{th:BCEmomlogID} and Proposition \ref{prop:wtslogCAvsmomlogCM}. \  For general $W_\alpha$, consider the shift $W_{\beta}$ with weight sequence $(\beta_n)$ satisfying $\beta_n = \frac{\alpha_n}{\|W_\alpha\|}$. \  The weights $\beta_n$ are clearly log completely alternating, so $W_{\beta}$ is $\mathcal{MID}$. \  Surely if $W_{\beta}^{(p)}$ is subnormal for some $p > 0$ it is immediate from the definitions that $W_{\alpha}^{(p)} = \|W_\alpha\|^{p} W_{\beta}^{(p)}$ is subnormal, and moment infinite divisibility follows for $W_\alpha$. \ 
\end{proof}

This clearly differs for moment conditions, since the moments of a scaled shift are not the straightforward scaling of the original moments. \  Alternatively, one may note that the Agler conditions for subnormality (which conditions are of moment type) do assume that the operator is a contraction.

Recall that from \cite[Prop 2.4]{BCE} that for positive term sequences completely alternating implies log completely alternating. \  The containment reverses when we move to completely monotone and log completely monotone (contractive) sequences.

\begin{cor}
We have for contractive (positive term) sequences that log completely monotone implies completely monotone.
\end{cor}

\begin{proof}
By the above, log completely monotone for a contractive sequence yields that the shift is $\mathcal{MID}$ and in particular subnormal, which means that its moment sequence is completely monotone.
\end{proof}

Note that implicit in the above is that a ``strong condition'' at the level of weights (log completely alternating) is equivalent to a ``strong condition'' at the level of moments (log completely monotone). \  The situation is well-captured by the following diagram:
\vspace{.2in}

%\begin{array}{ccc}
%\begin{boxedverbatim}
%Weights squared completely alternating (shift is $\mathcal{MID} \; + \;$ additional property
%\end{boxedverbatim}

$\begin{array}{ccc} 
\framebox[2.41in][c]{
$\begin{array}{c}
\textrm{Weights squared} \\
\textrm{completely alternating} \\
(\textrm{shift is } \mathcal{MID} \; + \; \textrm{add. property})
\end{array}$
} && \\
\mbox{\Huge{$\Downarrow$}} \hspace*{.2in}  \mbox{\Huge{$\not\Uparrow$}} && \\

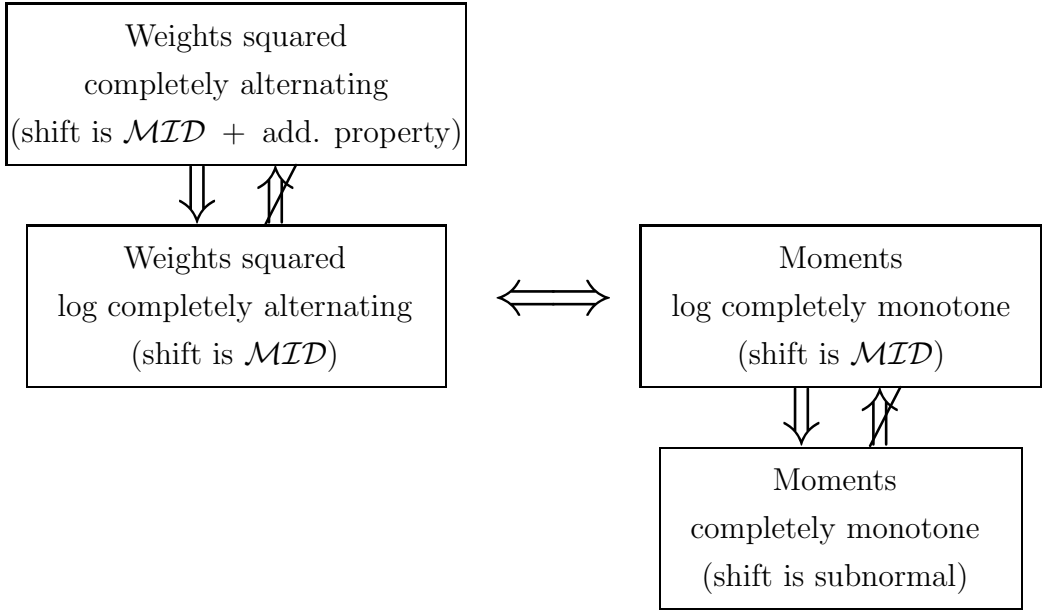
\framebox[2.2in][l]{
$\begin{array}{c}
\textrm{Weights squared} \\
\textrm{log completely alternating} \\
(\textrm{shift is } \mathcal{MID})
\end{array}$ 
}  & \mbox{\Huge{$\Longleftrightarrow$}} & 
\framebox[2.1in][l]{
$\begin{array}{c}
\textrm{Moments} \\
\textrm{log completely monotone} \\
(\textrm{shift is } \mathcal{MID})
\end{array}$
} \\
&&  \mbox{\Huge{$\Downarrow$}} \hspace*{.2in}  \mbox{\Huge{$\not\Uparrow$}}   \label{diagram}  \\
&& \framebox[1.9in][l]{
$\begin{array}{c}
\textrm{Moments} \\
\textrm{completely monotone} \\
(\textrm{shift is subnormal}) 
\end{array}$
}	
\end{array}$
\bigskip
\bigskip

We recall that a completely alternating sequence is log completely alternating (and the reverse need not hold), and so weights completely alternating is a sufficient condition for moment infinite divisibility. \ On the other hand, it is not known what additional property the stronger condition gives for a weighted shift.

\begin{remark}
{\rm We pause for a moment to consider some alternative potential definitions of $n$--log alternating for sequences. \  Observe that $1$--log alternating for a sequence $(a_n)$ is $\log a_n - \log a_{n+1} \leq 0$, which is simply $a_n \leq a_{n+1}$ and which is familiar for weights as non-decreasing (i.e., the shift is hyponormal). \  And $2$--log alternating is $\log a_n - 2 \log a_{n+1} + \log a_{n+2} \leq 0$, which translates to
$$\log\left(\frac{a_n a_{n+2}}{a_{n+1}^2}\right) \leq 0,$$
that is,
$$\frac{a_n a_{n+2}}{a_{n+1}^2} \leq 1,$$
or
$$a_n a_{n+2} \leq a_{n+1}^2.$$
This property has been studied by many authors under the name ``log-concave.''  It has the companion condition ``log-convex'' which is, as expected,
$$a_n a_{n+2} \geq a_{n+1}^2.$$
The papers \cite{StR} and \cite{WZ} contain some references to these studies.

Other authors have considered generalizations of this to other $n$ (again, see \cite{WZ})  but they are not the generalizations provided by the $n$-log alternating or $n$--log monotone ones we consider. \  These authors generalize $2$--log-concave as follows:  define a map $\mathcal{L}$ on sequences by $\mathcal{L}(a_k) = b_k$ where $b_0 = a_0^2$ and $b_k = a_k^2 - a_{k+1}a_{k-1}$ for $k \geq 1$. \  It is obvious that a sequence $(a_k)$ is log-concave $=$ $2$--log-concave if the sequence $\mathcal{L}(a_k)$ is non-negative. \ They then define $n$--log-concave by iteration:  the sequence $(a_k)$ is said to be $n$--log-concave if the sequence $\mathcal{L}^{n-1}(a_k)$ is non-negative. \  There is a similarly defined companion notion of $n$--log-convex. \ It is simple to check that these are not the same as our version (in either the concave or convex case) if $n \geq 3$.  

One result is that every Stieltjes moment sequence is infinitely log-convex, meaning, of course, that it is $n$--log-convex for every $n$, in \cite[Theorem 2.3]{WZ}. \ (Recall that a sequence of positive numbers is a Stieltjes moment sequence if it consists of the power moments of a positive Borel measure on $[0,+\infty)$.) \ There is no example given of an infinite sequence which is infinitely log-concave. \  Their version of infinitely log-convex is surely different from that used here (completely log monotone)  since it includes all Stieltjes moment sequences (and an even larger class). \  So it is perhaps surprising that apparently it is ``hard'' to be infinitely log-concave.

Observe that the weights squared of the Bergman shift (a shift with extremely good properties) are not even $3$--log-concave according to their generalization.} 
\end{remark}

%%%%%%%%%%%%%%%%%%%%%%%%
%%%%%%%%%%%%%%%%%%%%%%%%
%%%%%%%%%%%%%%%%%%%%%%%%
%%%%%%%%%%%%%%%%%%%%%%%%
%%%%%%%%%%%%%%%%%%%%%%%%
%%%%%%%%%%%%%%%%%%%%%%%%

\section{Consequences for the Class $\mathcal{MID}$}  \label{Sect3}

We seek to show in this section that moment infinite divisibility of a weighted shift is robust under a variety of operations (in particular the Aluthge transform) and is rigid in certain senses.

%%%%%%%%%%%%%%%%%%%%%%%%
%%%%%%%%%%%%%%%%%%%%%%%%
%%%%%%%%%%%%%%%%%%%%%%%%

\subsection{A quotient operation and the Aluthge Transform} \label{quotient}

A starting point is a very useful and somewhat surprising result about log completely alternating sequences.

\begin{prop}  \label{prop:spacedseqquotient}
Suppose $a = (a_n)$ is a log completely alternating sequence with positive terms, and let $N \in \mathbb{N}$. \  Then the sequence $b$ defined by $b_n := \frac{a_n}{a_{n+N}}$ is log completely alternating.
\end{prop}

\begin{proof}[(Sketch of Proof)]
Choose $N$ as in the statement. \  Use $\nabla$ for the differences of the original sequence $ \log a$ and $\tilde{\nabla}$ for the differences of $ \log b$. \  It is a computation to show that for any $k$,
$$
(\tilde{\nabla}^{1})_k = (\nabla^{2})_k + (\nabla^{2})_{k+1} + \cdots + (\nabla^{2})_{k+N-1}.
$$
(An indication of what is needed is to note that, with $N = 3$,
$$\begin{array}{lllllll}
\log \frac{a_0}{a_3} - \log \frac{a_1}{a_4} &=& &\log a_0 &- \log a_1 &&- \log a_3 + \log a_4 \\
&=& &\log a_0 &- 2 \log a_1 &+ \log a_2 &+ \\
&& &&+ \log a_1 &- 2 \log a_2 &+ \log a_3 \hspace{.1in}+ \\
&&  &&&+ \log a_2 &- 2 \log a_3 + \log a_4,  \\
\end{array}
$$
which is precisely $(\tilde{\nabla}^{1})_0 = (\nabla^{2})_0 + (\nabla^{2})_1 + (\nabla^{2})_2$.)

But then
\begin{eqnarray*}
(\tilde{\nabla}^{2})_k &=& (\tilde{\nabla}^{1})_k - (\tilde{\nabla}^{1})_{k+1} \\
&=& (\nabla^{2})_k + (\nabla^{2})_{k+1} + \cdots + (\nabla^{2})_{k+N-1}   \\
&&\hspace*{.1in}- ((\nabla^{2})_{k+1} + (\nabla^{2})_{k+2} + \cdots + (\nabla^{2})_{k+N}) \\
&=& ((\nabla^{2})_k - (\nabla^{2})_{k+1}) + ((\nabla^{2})_{k+1}- (\nabla^{2})_{k+2}) + \cdots \\
&& \hspace*{.1in}+  ((\nabla^{2})_{k+N-1} - (\nabla^{2})_{k+N}) \\
&=& (\nabla^{3})_k + (\nabla^{3})_{k+1} + \cdots + (\nabla^{3})_{k+N-1}. \\
\end{eqnarray*}
Extending this result to higher order differences by induction, and using that $a = (a_n)$ is log completely alternating, we obtain that $b = (b_n)$ is log completely alternating.  
\end{proof}

Note that there is a version of this at the level of ``$n$'' as opposed to at the level of ``all'';  indeed, if $(a_n)$ is ($k+1$)--log alternating then $(b_n)$ is $k$--log alternating; observe that there is no assumption that the sequence is bounded. \  A corollary is immediate.

\begin{thm}  \label{cor:simplyweightquotient}
Let $a = (a_n)$ be a non-decreasing sequence of positive numbers. \  Then $a$ is log completely alternating if and only if the quotient sequence $b$ with $b_n = \frac{a_n}{a_{n+1}}$ is log completely alternating.
\end{thm}

\begin{proof}
For the only if, it is easy to see that the test for $n$--log alternating for $b$ is the test for ($n+1$)--log alternating for 
$a$, and we get $1$--log alternating for $a$ since the $a_n$ are non-decreasing. \ The forward direction follows from Proposition \ref{prop:spacedseqquotient}.  
\end{proof}

\begin{cor} \label{cor42}
Let $W_\alpha$ be a contractive hyponormal weighted shift. \ Then $W_\alpha$ is $\mathcal{MID}$ if and only if $W_\beta$ is $\mathcal{MID}$, where $\beta$ is the quotient sequence $\beta_n = \frac{\alpha_n}{\alpha_{n+1}}$.
\end{cor}

Observe that some assumption on the sequence $a$ (respectively, the weight sequence $\alpha$) in Theorem \ref{cor:simplyweightquotient} (resp. Corollary \ref{cor42}) is required, because if $b$ is the moment sequence of the Dirichlet shift (the canonical completely hyperexpansive operator and $2$-isometry which is not subnormal), a sequence $a$ whose quotient is $b$ will be decreasing. \ Indeed, record for future use that if $b$ is the quotient sequence of $a$, then
\begin{equation}   \label{eq:relwtsquoseq}
a_n = \frac{a_0}{\prod_{i=0}^{n-1} b_i}.
\end{equation}
We record also that if $\alpha$ is a weight sequence and $\beta$ its quotient (weight) sequence, then the moment sequences satisfy
$$
\gamma^{(\beta)}_n = \gamma^{(\alpha)}_1 \frac{\gamma^{(\alpha)}_n}{\gamma^{(\alpha)}_{n+1}}.
$$

In light of these, clearly one could replace the assumption above that the sequence $a$ is increasing with the assumption that the sequence $b$ is bounded above by $1$, or, in the corollary, that $W_\beta$ is a contraction.

We pause briefly to to exhibit an example of an $\mathcal{MID}$ shift arising from the quotient operation of Corollary 
\ref{cor42}. 
   
\medskip
\begin{ex} \label{ex2}
The weighted shift with weights 
$$
\alpha_n:=\left(\frac{(n+1)(n+3)}{(n+2)^2}\right)^{1/2}
$$
is $\mathcal{MID}$ and hence, in particular, subnormal. \ In fact, this is the quotient operation of Corollary \ref{cor:simplyweightquotient} applied to the (moment infinitely divisible) Bergman shift $B$. 
\end{ex}
\medskip

\noindent We remark that for the weighted shift in Example \ref{ex2} we know both the Berger measure for the moments ($\mu(t) = \frac{1}{2} \chi_{[0,1]}(t) dt + \frac{1}{2} \delta_1(t) dt$) and a Levy-Khinchin measure (see \cite[Chapter 4, Prop. 6.12]{BCR}) for the weights squared ($- \ln t \chi_{[0,1]}(t) dt$) where $\delta_1$ denotes the usual Dirac point mass at $1$. \ (These may be checked by straightforward computations once the candidate is guessed;  note that in this case there is a Levy-Khinchin measure for the weights squared because they are not only log completely alternating but actually completely alternating.) \ And note also that we may do the same thing, or do the quotient operation multiple times, for any of the Agler shifts.

With Proposition \ref{prop:spacedseqquotient} in hand, we may give an example of the ``robustness'' of moment infinite divisibility of weighted shifts. \ Recall that if $W_\alpha$ and $W_\beta$ are subnormal (respectively, $\mathcal{MID}$) weighted shifts then the ``Schur product'' shift with weights $(\alpha_n \beta_n)$ is subnormal (respectively, $\mathcal{MID}$); this follows from a $k$-hyponormality approach and the standard positivity fact about Schur products of matrices.

\begin{thm}     \label{th:WvsATW}
Suppose $W_\alpha$ is a contractive weighted shift whose weights approach a limit (in particular, if $W_\alpha$ is hyponormal). \ Then the Aluthge transform $AT(W_\alpha)$ is $\mathcal{MID}$ if and only if $W_\alpha$ is.
\end{thm}

\begin{proof}[Sketch of Proof]
If $W_\alpha$ is $\mathcal{MID}$, its weights squared sequence is log completely alternating, and it follows from a trivial computation that the weight sequence for $AT(W_\alpha)$,
\begin{equation}    \label{seq:original}
\sqrt{\alpha_0 \alpha_1}, \sqrt{\alpha_1 \alpha_2},\sqrt{\alpha_2 \alpha_3},\sqrt{\alpha_3 \alpha_4},\sqrt{\alpha_4 \alpha_5},\sqrt{\alpha_5 \alpha_6},\sqrt{\alpha_6 \alpha_7}, \ldots ,
\end{equation}
is log completely alternating.

For the other direction, consider the weight sequence for $AT(W_\alpha)$ as in \eqref{seq:original}. \
By taking the usual quotient of successive entries (Corollary \ref{cor:simplyweightquotient}) we obtain that
\begin{equation}
\frac{\sqrt{\alpha_0}}{\sqrt{\alpha_2}}, \frac{\sqrt{\alpha_1}}{\sqrt{\alpha_3}},\frac{\sqrt{\alpha_2}}{\sqrt{\alpha_4}},
\frac{\sqrt{\alpha_3}}{\sqrt{\alpha_5}},
\frac{\sqrt{\alpha_4}}{\sqrt{\alpha_6}} \ldots
\end{equation}
is log completely alternating. \  Taking the Schur product of this last and the tail of \eqref{seq:original} beginning with $\sqrt{\alpha_2 \alpha_3},\sqrt{\alpha_3 \alpha_4}, \ldots$
we obtain that
\begin{equation}    \label{seq:firstprod}
\sqrt{\alpha_0 \alpha_3}, \sqrt{\alpha_1 \alpha_4},\sqrt{\alpha_2 \alpha_5},\sqrt{\alpha_3 \alpha_6},\sqrt{\alpha_4 \alpha_7},\sqrt{\alpha_5 \alpha_8},\sqrt{\alpha_6 \alpha_9}, \ldots
\end{equation}
is log completely alternating. \ Use the operation in Proposition \ref{prop:spacedseqquotient} with $N=3$ applied to  the sequence in \eqref{seq:firstprod} to obtain that
\begin{equation}  \label{seq:secondquo}
\frac{\sqrt{\alpha_0}}{\sqrt{\alpha_6}}, \frac{\sqrt{\alpha_1}}{\sqrt{\alpha_7}},\frac{\sqrt{\alpha_2}}{\sqrt{\alpha_8}},
\frac{\sqrt{\alpha_3}}{\sqrt{\alpha_9}},
\frac{\sqrt{\alpha_4}}{\sqrt{\alpha_{10}}} \ldots
\end{equation}
is log completely alternating. \ Form the Schur product of this with the tail of \eqref{seq:original} beginning with $\sqrt{\alpha_6 \alpha_7},\sqrt{\alpha_7 \alpha_8}, \ldots$. \ This yields that
\begin{equation}    \label{seq:secondprod}
\sqrt{\alpha_0 \alpha_7}, \sqrt{\alpha_1 \alpha_8},\sqrt{\alpha_2 \alpha_9},\sqrt{\alpha_3 \alpha_{10}},\sqrt{\alpha_4 \alpha_{11}},\sqrt{\alpha_5 \alpha_{12}},\sqrt{\alpha_6 \alpha_{13}}, \sqrt{\alpha_7 \alpha_{14}} \ldots
\end{equation}
is log completely alternating.
Repeat the operation from Proposition \ref{prop:spacedseqquotient} on the sequence in \eqref{seq:secondprod}, this time with $N=7$, to yield that the sequence
$$\frac{\sqrt{\alpha_0 \alpha_7}}{\sqrt{\alpha_7 \alpha_{14}}}, \frac{\sqrt{\alpha_1 \alpha_8}}{\sqrt{\alpha_8 \alpha_{15}}}, \ldots,$$
which is
\begin{equation}
\frac{\sqrt{\alpha_0}}{\sqrt{\alpha_{14}}}, \frac{\sqrt{\alpha_1}}{\sqrt{\alpha_{15}}},\frac{\sqrt{\alpha_2}}{\sqrt{\alpha_{16}}},
\frac{\sqrt{\alpha_3}}{\sqrt{\alpha_{17}}},
\frac{\sqrt{\alpha_4}}{\sqrt{\alpha_{18}}} \ldots,
\end{equation}
is log completely alternating.

By continuing this process indefinitely, one may obtain that for some $M$ as large as desired the sequence
\begin{equation}
\frac{\sqrt{\alpha_0}}{\sqrt{\alpha_{M}}}, \frac{\sqrt{\alpha_1}}{\sqrt{\alpha_{M+1}}},\frac{\sqrt{\alpha_2}}{\sqrt{\alpha_{M+2}}},
\frac{\sqrt{\alpha_3}}{\sqrt{\alpha_{M+3}}},
\frac{\sqrt{\alpha_4}}{\sqrt{\alpha_{M+4}}} \ldots,
\end{equation}
is log completely alternating (there is no claim that one can do this for any $M$, but only that we may do it so as to obtain some $M$ large).

Now consider some test for the sequence $\sqrt{\alpha_0}, \sqrt{\alpha_1}, \sqrt{\alpha_2}, \ldots$ to be (say) $2$--log alternating. \  This is (for example)
$$\log \sqrt{\alpha_0} - 2 \log \sqrt{\alpha_1} + \log \sqrt{\alpha_2} \leq 0,$$
and suppose for a contradiction that the left-hand side is strictly positive. \  We have that
\begin{eqnarray*}
0 &\geq& \log \frac{\sqrt{\alpha_0}}{\sqrt{\alpha_M}} - 2\log \frac{\sqrt{\alpha_1}}{\sqrt{\alpha_{M+1}}} + \log \frac{\sqrt{\alpha_2}}{\sqrt{\alpha_{M+2}}}  \\
&=& \log \sqrt{\alpha_0} - 2 \log \sqrt{\alpha_1} + \log \sqrt{\alpha_2} - (\log \sqrt{\alpha_M} - 2 \log \sqrt{\alpha_{M+1}} + \log \sqrt{\alpha_{M+2}})
\end{eqnarray*}
for large $M$ as above. \  But since the weights approach a limit, say $L$, the term $(\log \sqrt{\alpha_M} - 2 \log \sqrt{\alpha_{M+1}} + \log \sqrt{\alpha_{M+2}})$ approaches $0$, and taking into account the assumed signs this yields a contradiction.  
\end{proof}

\begin{cor} \label{cor45}
Suppose $W_\alpha$ is a contractive weighted shift whose weights approach a limit (in particular, if $W_\alpha$ is hyponormal). \  Then $W_\alpha$ is $\mathcal{MID}$ if and only if any of the iterated Aluthge transforms of $W_\alpha$ is $\mathcal{MID}$, and in this case all the iterated Aluthge transforms are $\mathcal{MID}$.
\end{cor}

\begin{proof}
It is routine that if the weights of $W_\alpha$ approach a limit, then so do the weights of any of the iterated Aluthge transforms.
\end{proof}

For the next result, we recall that the Schur square root of a weighted shift $W_{\alpha}$ is defined as the weighted shift with weight sequence $\{\sqrt{\alpha_n} \}_{n=0}^{\infty}$.
 
\begin{cor} \label{cor47}
Suppose $W_\alpha$ is a contractive hyponormal weighted shift. \ Then $W_{\alpha}^{(1/2)}$, the Schur square root of $W_\alpha$, is $\mathcal{MID}$ if and only if the Aluthge transform of $W_\alpha$ is $\mathcal{MID}$.
\end{cor}

\begin{proof}
Surely $W_\alpha$ is $\mathcal{MID}$ if and only if $W_\alpha^{(1/2)}$ is.  
\end{proof}

We know in general that if the Schur square root of any shift is subnormal, then the Aluthge transform is, but the reverse question is essentially open.

Recall that there are generalizations of the Aluthge transform which are asymmetric: while the usual Aluthge transform of an operator $T$ is $AT(T) = |T|^{1/2} U |T|^{1/2}$ where $T = U |T|$ is the standard polar decomposition of $T$, we may consider, for any $0 < q < 1$, $AT_q(T) = |T|^{q} U |T|^{1-q}$. \ It is straightforward to check that for a weighted shift $W_\alpha$ with weight sequence $\alpha: \alpha_0, \alpha_1, \ldots$ one has
$AT_q(W_\alpha)$  is again a weighted shift with the weight sequence
$$
\alpha_0^{1-q}\alpha_1^{q},\, \alpha_1^{1-q}\alpha_2^{q},\, \alpha_2^{1-q}\alpha_3^{q},\, \ldots.
$$
If $W_\alpha$ is $\mathcal{MID}$, it is immediate that $AT_q(W_\alpha)$ is $\mathcal{MID}$ by considering Schur products. \ With some preliminary work we may obtain a partial converse.

\begin{definition} Given the sequences $a = (a_n)_{n=0}^\infty$ and $b = (b_n)_{n=0}^\infty$, we say that $a$ \textit{negatively dominates} $b$, and write $a \preceq b$, if
$$(\nabla^{m} a)_j \leq (\nabla^{m} b)_j, \qquad (m \in \mathbb{N}_0, j \in \mathbb{N}_0).$$
\end{definition}

The next two results are just computations.

\begin{prop}
Suppose $a = (a_n)_{n=0}^\infty$ is completely alternating and $a \preceq b = (b_n)_{n=0}^\infty$, and let $r$ satisfy $0 \le r \le 1$. \  Then $a-r b = (a_n - r b_n)_{n=0}^\infty$ is completely alternating.
\end{prop}

\begin{cor}
If $(\ln a_n)_{n=0}^\infty$ and $(\ln b_n)_{n=0}^\infty$ are defined, $(\ln a_n)_{n=0}^\infty$ is completely alternating and negatively dominates $(\ln b_n)_{n=0}^\infty$, and $0 \le r \le 1$, then $(\ln(\frac{a_n}{b_n^r}))$ is completely alternating.
\end{cor}

\begin{prop}
Suppose $a = (a_n)_{n=0}^\infty$ is completely alternating, and $0 \leq r \leq 1$. \  Let $N \in \mathbb{N}_0$. \  Then $a \preceq b$ where $b$ is defined by $b_n := r a_{n+N}$.
\end{prop}

\begin{proof}
For any $m$, one readily computes that since $(\nabla^{m+1} a)_j \leq 0$ for all $j = 0, 1, \ldots$, $(\nabla^{m} a)_j \leq (\nabla^{m} a)_{j+1}$ for all $j = 0, 1, \ldots$. \  It then follows that $(\nabla^{m} a)_j \leq (\nabla^{m} a)_{j+N}$ for all $j = 0, 1, \ldots$, so $a$ dominates $b$ in the case $r= 1$, and use the corollary.  
\end{proof}

We then have the following.

\begin{cor}  \label{cor:forAsymLCA}
If $(\ln a_n)_{n=0}^\infty$ is defined and completely alternating, then for any $N \in \mathbb{N}_0$ and any $0 \leq r \leq 1$, $\left( \ln \left(\frac{a_n}{a_{n+N}^r}\right)\right)_{n=0}^\infty$ is completely alternating.
\end{cor}

We may now give the promised partial generalization of the harder direction of Theorem \ref{th:WvsATW} to some asymmetrical Aluthge transforms.

\begin{thm} \label{thm312}
Suppose $W_\alpha$ is a weighted shift whose weights approach a limit (in particular, if $W_\alpha$ is hyponormal). \ Suppose $0 \le q \le 1/2$. \ If the generalized Aluthge transform $AT_q(W_\alpha)$ is $\mathcal{MID}$ then so is $W_\alpha$.
\end{thm}

\begin{proof}
For ease of exposition, we will consider the case $q = 1/3$;  the modifications for the general case are routine and consist primarily in realizing that certain constants are bounded above (in fact, by $1$). \  The argument is essentially that of the proof of the result for $q = 1/2$ with modest alterations.

Assume that $AT_{\frac{1}{3}}(W_\alpha)$ is $\mathcal{MID}$, so the sequence
\begin{equation}  \label{eq:qLCA1}
\alpha_0^{\frac{2}{3}} \alpha_1^{\frac{1}{3}}, \alpha_1^{\frac{2}{3}} \alpha_2^{\frac{1}{3}}, \alpha_2^{\frac{2}{3}} \alpha_3^{\frac{1}{3}}, \ldots
\end{equation}
is log completely alternating. \  The goal is to obtain that $\alpha_0, \alpha_1, \ldots$ is log completely alternating, and it is clearly sufficient to obtain that $\alpha_0^{\frac{2}{3}}, \alpha_1^{\frac{2}{3}}, \ldots$ is log completely alternating.
From \eqref{eq:qLCA1}, by using Corollary \ref{cor:forAsymLCA} with $N = 1$ and $r = 1/2$, we obtain that
\begin{equation}
\frac{\alpha_0^{\frac{2}{3}} \alpha_1^{\frac{1}{3}}}{\left(\alpha_1^{\frac{2}{3}} \alpha_2^{\frac{1}{3}}\right)^{1/2}}, \frac{\alpha_1^{\frac{2}{3}} \alpha_2^{\frac{1}{3}}}{\left(\alpha_2^{\frac{2}{3}} \alpha_3^{\frac{1}{3}}\right)^{1/2}}, \frac{\alpha_2^{\frac{2}{3}} \alpha_3^{\frac{1}{3}}}{\left(\alpha_3^{\frac{2}{3}} \alpha_4^{\frac{1}{3}}\right)^{1/2}}, \ldots  = \frac{\alpha_0^{\frac{2}{3}}}{\alpha_2^{\frac{1}{6}}},
\frac{\alpha_1^{\frac{2}{3}}}{\alpha_3^{\frac{1}{6}}},
\frac{\alpha_2^{\frac{2}{3}}}{\alpha_4^{\frac{1}{6}}}, \ldots
\end{equation}
is log completely alternating. \  Upon taking the Schur product of this with
$$(\alpha_2^{2/3} \alpha_3^{1/3})^\frac{1}{4}, (\alpha_3^{2/3} \alpha_4^{1/3})^\frac{1}{4}, \ldots$$
 (which is log completely alternating as it is a tail of \eqref{eq:qLCA1} to a power) we obtain that
\begin{equation}  \label{eq:qLCA2}
\alpha_0^{\frac{2}{3}} \alpha_3^{\frac{1}{12}}, \alpha_1^{\frac{2}{3}} \alpha_4^{\frac{1}{12}}, \alpha_2^{\frac{2}{3}} \alpha_5^{\frac{1}{12}}, \ldots
\end{equation}
is log completely alternating. \  Using this last sequence and Corollary \ref{cor:forAsymLCA} with $N = 3$ and $r = 1/8$ we obtain that
\begin{equation}
\frac{\alpha_0^{\frac{2}{3}}}{\alpha_6^{\frac{1}{96}}},
\frac{\alpha_1^{\frac{2}{3}}}{\alpha_7^{\frac{1}{96}}},
\frac{\alpha_2^{\frac{2}{3}}}{\alpha_8^{\frac{1}{96}}}, \ldots
\end{equation}
is log completely alternating, and using a Schur product of this with
$$(\alpha_6^{2/3} \alpha_7^{1/3})^\frac{1}{64}, (\alpha_7^{2/3} \alpha_8^{1/3})^\frac{1}{64}, \ldots$$
we obtain that
\begin{equation}
\alpha_0^{\frac{2}{3}} \alpha_7^{\frac{1}{192}}, \alpha_1^{\frac{2}{3}} \alpha_8^{\frac{1}{192}}, \alpha_2^{\frac{2}{3}} \alpha_9^{\frac{1}{192}}, \ldots
\end{equation}
is log completely alternating. \  The rest of the argument is as in the proof of Theorem \ref{th:WvsATW};  the only point to note is that the exponents of the second terms in the resulting collection of sequences are always less than $1$.  
\end{proof}

\begin{remark}
{\rm We do not know the status of the obvious question concerning $AT_q$ when $q > 1/2$.} 
\end{remark}

Returning to consideration of the standard Aluthge transform (which we now view as a mapping $AT$ from operators to operators), the result in Theorem \ref{th:WvsATW} taken with the ``scaling'' discussion makes it clear that $AT(\mathcal{MID}) \subseteq (\mathcal{MID})$. \ We consider this mapping briefly as one from weighted shifts to weighted shifts, and what is in some sense the ``inverse'' mapping in which, given a target shift $W_\beta$, we produce a shift $W_\alpha$ such that $AT(W_\alpha) = W_\beta$.

Recall that the relationship between the weights is that if $W_\alpha$ has weight sequence $(\alpha_n)_{n=0}^\infty$ then the weight sequence $\beta$ is
$$
\sqrt{\alpha_0 \alpha_1}, \sqrt{\alpha_1 \alpha_2}, \sqrt{\alpha_2 \alpha_3}, \ldots.
$$
One readily computes from this that in order to have $AT(W_\alpha) = W_\beta$, the weight $\alpha_0$ is a free parameter and all the remaining $\alpha_n$ are then fixed:
\begin{equation}    \label{eq:alphsintermsofbetas}
\alpha_n = \left\{
\begin{array}{cc}
\frac{\prod_{j=0}^{\frac{n-1}{2}} \beta_{2j}^2}{\alpha_0 \prod_{j=1}^{\frac{n-1}{2}} \beta_{2j-1}^2}, & n \, \,  \mbox{\rm odd,} \\
& \\
\frac{\alpha_0 \prod_{j=1}^{\frac{n}{2}}\beta_{2j-1}^2}{\prod_{j=0}^{\frac{n}{2}-1} \beta_{2j}^2}, & n \, \,  \mbox{\rm even.}\\
\end{array}\right.
\end{equation}

Let $\mathcal{L}$ denote the collection of weighted shifts whose weights approach a non-zero limit. \  Let $AT$ denote the Aluthge transform map on the collection of (bounded) weighted shifts. \  We have the following lemma.

\begin{lem}  The map $AT|_\mathcal{L}$ is injective.
\end{lem}

\begin{proof}
Let $W_\beta$ be some bounded shift such that $AT(W_\alpha) = AT(W_{\hat{\alpha}}) = W_\beta$ for some weight sequences $\alpha$ and $\hat{\alpha}$. \  Write $\hat{\alpha}_0 = x \alpha_0$ for some $x > 0$, and suppose $\lim_{n \rightarrow \infty} \alpha_n = L \neq 0$. \  Using the equations above, one obtains that
$$\lim_{j \rightarrow \infty} \hat{\alpha}_{2j} = x L$$
and
$$\lim_{j \rightarrow \infty} \hat{\alpha}_{2j-1} = \frac{1}{x} L.$$
Since the sequence $(\hat{\alpha})_{n=0}^\infty$ has a limit, and $L \neq 0$, clearly $x = 1$ and $(\alpha)_{n=0}^\infty$ and $(\hat{\alpha})_{n=0}^\infty$ coincide.  
\end{proof}

%%%%%%%%%%%%%%%%%%
%%%%%%%%%%%%%%%%%%
%%%%%%%%%%%%%%%%%%

It is reasonable to consider the map $AT$ now restricted to $\mathcal{MID}$ weighted shifts, since we know $AT(\mathcal{MID}) \subseteq \mathcal{MID}$. \  Of course $\mathcal{MID}$ weighted shifts have weights approaching a limit (since they are bounded and hyponormal) so the map is injective. \  For this smaller class, the result above can be improved.

\begin{thm} \label{prop:ATbijective}
The map $AT|_{\mathcal{MID}}$ is a bijection from ${\mathcal{MID}}$ to ${\mathcal{MID}}$.
\end{thm}

\begin{proof} 
From the previous proposition and Theorem \ref{th:WvsATW} we have that the map is injective and that its range is contained in $\mathcal{MID}$. \  For surjectivity, take first the case in which some $W_\beta$ in $\mathcal{MID}$ has weights increasing to $1$. \ Consider the expression for the sequence of odd weights  in \eqref{eq:alphsintermsofbetas} for a candidate pre-image $W_\alpha$, removing for a moment the parameter $\alpha_0$. \  It is easy to see that the resulting sequence is increasing since one moves from one term to the next by multiplying by some $\frac{\beta_{2j}^2}{\beta_{2j-1}^2}$ and the $\beta_k$ are increasing. \  As well, the sequence is bounded above by one, since it is a product of terms of the form $\frac{\beta_{2j}^2}{\beta_{2j+1}^2}$ and a single term $\beta_{2k} <1$. \ Therefore it has a limit $L$ satisfying $0 < L \leq 1$. \  One may similarly see that the sequence for the even weights in \eqref{eq:alphsintermsofbetas} (again removing for the moment $\alpha_0$) is also increasing, bounded above by $\frac{1}{\beta_0^2}$, and therefore has a limit, which turns out to be $\frac{1}{L}$.

Choosing $\alpha_0 = L$, it is easy to see that the candidate sequence $\alpha$ has the limit $1$ and satisfies $AT(W_\alpha) = W_\beta$;  citing again Theorem \ref{th:WvsATW}, $W_\alpha \in \mathcal{MID}$.  

\end{proof}

We pause for a moment to record the pre-images under $AT$ of the ($\mathcal{MID}$) Agler shifts:  recall that the $k$-th Agler shift $A_k$ is that with weight sequence $\{\beta(k)_n\}_{n=0}^\infty$ with $\beta(k)_n = \sqrt{\frac{n+1}{n+k}}$ for $k = 2, 3, \ldots$ and $n = 0, 1, \ldots$. \  (These appear in the foundational paper \cite{Ag} on $n$-contractivity;  the most familiar is 
$A_2$ which is the Bergman shift.)

\begin{prop} \label{prop417} 
For $k = 2, 3, \ldots,$ the Agler shift $A_k$ is the image under $AT$ of the shift $W_{\alpha(k)}$ with initial weight
\begin{equation} \label{eq:alpha0forAk}
\alpha(k)_0 = \ddfrac{\Gamma \left(\frac{k}{2}\right)}{\sqrt{\pi } \Gamma \left(\frac{k+1}{2}\right)}
\end{equation}
and
\begin{equation}\label{eq:alphanforAgk}
\alpha(k)_n = \left\{
\begin{array}{cc}
\ddfrac{\Gamma \left(\frac{n}{2}+1\right) \Gamma \left(\frac{k+n}{2}\right)}{\alpha(k)_0 \Gamma \left(\frac{n+1}{2}\right) \Gamma \left(\frac{1}{2} (k+n+1)\right)}, & n \, \,  \mbox{\rm odd,} \\
& \\
\ddfrac{\alpha(k)_0 \sqrt{\pi}k (\frac{n}{2})! \Gamma\left(\frac{k+1}{2}\right)\Gamma\left(\frac{k+n}{2}\right)}{2 \Gamma\left(\frac{n+1}{2}\right)\Gamma\left(\frac{k}{2}+1\right)\Gamma\left(\frac{k+n+1}{2}\right)}, & n \, \,  \mbox{\rm even,}\\
\end{array}\right.
\end{equation}
\medskip
In particular, the
Bergman shift is the Aluthge transform of the shift with weight sequence beginning
$$
\frac{2}{\pi}, \frac{\pi}{4}, \frac{8}{3 \pi}, \frac{9 \pi}{32}, \frac{128}{45 \pi}, \ldots .
$$
\medskip
The initial weights for the pre-images of the first few Agler shifts are as follows:
$$\begin{array}{l|lllllllllllll}
k & 2 & 3 & 4 & 5 & 6 & 7 & 8 & 9 & 10 &11 & 12 & 13 & 14 \\ \hline
\mbox{\rm initial weight of} &&&&&&&&&&&&& \\
 \mbox{\rm pre-image shift of $A_k$} & \frac{2}{\pi} & \frac{1}{2} & \frac{4}{3 \pi} & \frac{3}{8} & \frac{16}{15 \pi} & \frac{5}{16} & \frac{32}{35 \pi} & \frac{35}{128} & \frac{256}{315 \pi} & \frac{63}{256} & \frac{512}{693 \pi} & \frac{231}{1024} & \frac{2048}{3003 \pi} \\
\end{array}$$
\end{prop}

\medskip

\begin{proof}
This is mainly a computation. \  First, one verifies that (in advance of knowing $\alpha(k)_0$) the expressions in \eqref{eq:alphsintermsofbetas} are as claimed in this particular case. \  It then remains, as in the proof of Proposition \ref{prop:ATbijective}, to evaluate the limit (say, of the odd weights);  this is accomplished using that
$$
\frac{\Gamma(z + a)}{\Gamma(z + b)} \sim z^{a-b}\sum_{k=0}^\infty \frac{G_k(a,b)}{z^k},
$$
where $G_0(a,b) = 1$ (cf. \cite[5.11.13]{DLMF}).  
\end{proof}

Note that it is straightforward to compute that the Aluthge transform of the Bergman shift $B$ is $A_3^{(1/2)}$, as might be suspected from the $n=3$ entry in the table.

We now return to the discussion of the Aluthge transform, and remark that if we seek the inverse Aluthge transform of the shift with weights
$$
\sqrt{\frac{2^{n+2}-2}{2^{n+2}-1}}
$$
(those arising from a certain moment sequence with countably atomic Berger measure mentioned in \cite{BCE}, the shift known to be $\mathcal{MID}$), the appropriate $\alpha_0$ is approximately
$$
\alpha_0 \sim 0.7421267409\ldots,
$$
but we do not know if the resulting shift has some simple form.

\smallskip

%%%%%%%%%%%%%%%%%%%%%%%%
%%%%%%%%%%%%%%%%%%%%%%%%
%%%%%%%%%%%%%%%%%%%%%%%%

\subsection{Robustness and rigidity} \label{robust}
The results in the previous subsection may be interpreted to say that the class of $\mathcal{MID}$ shifts is robust under some operations (quotients of weights, for example, as implicit in Proposition \ref{prop:spacedseqquotient}, or the Aluthge transform as in Theorem \ref{th:WvsATW}). \  We record next some further results in this family, some showing ``rigidity'' in the class of $\mathcal{MID}$ shifts, and consider as well back-step extensions and completions.

\begin{prop}  \label{prop:normlimsofinfdivis}
Suppose a shift $W_\alpha$ is the norm limit of contractive $\mathcal{MID}$ weighted shifts $W_{\delta^{(j)}}$. \  Then $W_\alpha$ is $\mathcal{MID}$.
\end{prop}

\begin{proof}
Any particular test for the weights $\alpha$ to be $n$--log alternating is the limit of the related tests for the $\delta^{(j)}$ to be $n$--log alternating.  
\end{proof}

\begin{prop}  \label{prop:wstarlimsofIDBerger}
Suppose a weighted shift $W_\alpha$ has a Berger measure $\mu$ which is the weak-* limit of the Berger measures $\mu_n$ of contractive $\mathcal{MID}$ weighted shifts $W_{\delta^{(n)}}$. \  Then $W_\alpha$ is contractive and $\mathcal{MID}$.
\end{prop}

\begin{proof}
Note that the weak-* limit will satisfy that, for any $m$,
$$\lim_{n \rightarrow \infty} \int_0^1 t^m d \mu_n(t) = \int_0^1 t^m d \mu(t).$$
Therefore moments, and hence weights, are approximated in the limit. \  It is clear that $\mu$ has support in $[0,1]$ and therefore $W_\alpha$ is a contraction. \  But now we may cite again that any particular test for the moments $\gamma$ to be $k$--log monotone is the limit of the related tests for the moments arising from $\mu_n$ to be $k$--log monotone.  
\end{proof}

There are results for powers and subshifts. \  First set some notation for subshifts of a weighted shift $W_\alpha$. \  A subshift is a shift whose weight sequence is of the form $\alpha \circ g$, where $g: \mathbb{N}_0 \rightarrow  \mathbb{N}_0$ is increasing and $\circ$ in this setting means composition. \  Call a subshift (or subsequence of weights) a $p$-subshift if $g$ is linear, of the form $g(n) = p n + k$ where $p,k \in \mathbb{N}_0$ and $p \geq 1$. \  Then we have the following.

\begin{prop}  \label{prop:regsubshiftsID}
Suppose $W_\alpha$ is $\mathcal{MID}$. \  Then any $p$-subshift of $W_\alpha$ is $\mathcal{MID}$ and so is any positive integer power (meaning ordinary, not Schur,  power) of $W_\alpha$. \  (Infinite divisibility of the latter has to be interpreted either viewing (say) $W_\alpha^2$ as the direct sum of two weighted shifts (which we know how to raise to Schur powers), or viewing $W_\alpha^n$ as a matrix and viewing a Schur $p$-th power as raising every entry in the matrix to the $p$-th power.)
\end{prop}

\smallskip

\begin{proof}
For the subshift claim we shall use the characterization in Theorem \ref{th:BCEmomlogID} that the weights squared be log completely alternating, and will use induction on $n$ for log $n$--alternating for all $n$. \  First, if $W_\alpha$ is $\mathcal{MID}$ it is subnormal (hence hyponormal, with non-decreasing weights); any $p$-subshift is therefore log $1$--alternating. \  Suppose now that for any $\mathcal{MID}$ weighted shift and any $p$-subshift $W_\beta$, the subshift is log $n$--alternating. \  Consider some expression $LT_\beta(n+1, k)$ to be tested for negativity. \ Apply the quotient operation of Proposition \ref{prop:spacedseqquotient} to the tail of the weight sequence $\alpha$ starting at $\alpha_k$ with $N$ set to $p$ to generate a new weighted shift $W_\delta$, which is $\mathcal{MID}$ by that proposition. \  It is an easy computation to show that $LT_\beta(n+1, k)$ is exactly the expression $LT_\delta(n, k)$ and is therefore non-positive by the induction hypothesis.

The result for powers follows from regarding some positive integer power of $W_\alpha$ as a direct sum of certain weighted shifts (see \cite{CPa}).  
\end{proof}

\begin{remark}
{\rm Observe that since we know that certain shifts $S(a,b,c,d)$ (as in Definition 2.5 of \cite{CPY}) are $\mathcal{MID}$ (Corollary 3.2 of \cite{BCE}) we obtain as a consequence that all $p$-subshifts of some $S(a,b,c,d)$ are $\mathcal{MID}$ and therefore subnormal, generalizing \cite[Theorem 2.13]{CPY}. \  Note also that, in that paper, the examples of subnormal shifts for which a $p$-subshift is not subnormal have finitely atomic Berger measures which are (as expected, cf. \cite{CuEx}) not $\mathcal{MID}$.

Evidence from \cite{Wol} suggests strongly that there is another approach to the proof of the $p$-subshift portion of Proposition \ref{prop:regsubshiftsID}, which is to show any of the $p$--subshift log-weight expressions to be tested for alternating-ness is in fact a sum, with positive coefficients, of log-weight expressions for the original shift (in fact, using the same ``$n$''). \  The coefficients are non-trivial, however, and we do not know how to succeed in this approach.

Finally, we have no example of a subnormal weighted shift each of whose $p$-subshifts is subnormal which is not, in fact, $\mathcal{MID}$.} 
\end{remark}

We now turn our attention to the matter of ``back-step extensions'' of $\mathcal{MID}$ weighted shifts; recall that this means prefixing a weight to the given weight sequence, and asking for properties of interest for the new shift. \  Here the motivating question is whether one may take an $\mathcal{MID}$ contractive weighted shift and produce a back-step extension which is subnormal but not $\mathcal{MID}$. \  Of course a back-step extension to subnormality is not always possible, and the relevant ``modulus of subnormality'' is to be found in \cite[Proposition 8]{Cu} in terms of an integral. \  Our only information on the question is the following.

\begin{prop}  Let $W$ be any regular subshift of one of the Agler shifts (which we know are $\mathcal{MID}$). \  Then any subnormal back-step extension of $W$ is $\mathcal{MID}$. \ (Note that if the ``previous weight'' of the subshift would require a negative index, no subnormal extension is possible at all.) 
\end{prop}

\begin{proof}
From \cite{CD} we obtain the Berger measure of the regular subshifts, and compute that if a subnormal back-step extension is possible at all then the cutoff is exactly the ``previous'' (omitted) weight in the subsequence in question;  a back-step extension with this weight which is a subnormal extension is therefore $\mathcal{MID}$. \  The only flexibility allowed to obtain a subnormal back-step extension is to decrease the zeroth weight below the cutoff. \ But we know that a reduction in the zeroth weight preserves moment infinite divisibility by \cite[Cor. 3.3]{BCE}. \  (Note that if the subshift begins with the zeroth weight, no subnormal extension is possible at all.)   
\end{proof}

We may also consider the analog of the classical completion question:  given some initial segment of weights, when it is possible to complete these to the weights for a subnormal (here, $\mathcal{MID}$) weighted shift?  There are two answers in the case of two given weights $\alpha_0 < \alpha_1$, one more interesting than the other.

For the uninteresting one, suppose we have $0 < \sqrt{a} < \sqrt{b} < 1$. \  First, it is easy to check that any two-atomic measure with an atom at zero gives an $\mathcal{MID}$ weighted shift. \ Express the measure as $c_0 \delta_0 + c_1\delta_r$, where $r \in (0, 1]$ and $c_0, c_1 \in (0,1)$ with $c_0 + c_1 = 1$. \  The weights $\alpha$ are $\sqrt{c_1 r}, \sqrt{r}, \sqrt{r}, \ldots$. \   Then $\alpha_0^2 - \alpha_1^2 = r(c_1-1) < 0$, and $\alpha_j^2 - \alpha_{j+1}^2 = 0$ for $j \geq 1$. \  Using this and the standard recurrence, one finds $\alpha_0^2 - 2 \alpha_1^2 + \alpha_2^2 = r(c_1-1) < 0$, and $\alpha_j^2 - 2 \alpha_{j+1}^2 + \alpha_{j+2}^2 = 0$, $j \geq 1$. \  Continuing, one produces that the weights-squared sequence is completely alternating, hence log completely alternating, and we have that the shift is $\mathcal{MID}$.

But to use this for the given $a$ and $b$, we must simply arrange $c_1$ and $r$ so that $a = c_1 r$ and $b = r$, and we have the completion to an $\mathcal{MID}$ weighted shift, but the completion is flat (see the discussion after Question \ref{question:3wtscomp}) and in that sense trivial.

There is a more interesting completion, and the idea is this:  first, any $a$ and $b$ as above such that both $a$ and $b$ are rational numbers have a nontrivial such completion, because $\sqrt{a}$ and $\sqrt{b}$ are two weights (of course not necessarily successive) in one of the Agler shifts. \  By taking the $p$-subshift of that Agler shift which has these as the first two weights, we get an $\mathcal{MID}$ shift because we know that $p$-subshifts of the Agler shifts are $\mathcal{MID}$. \  For more general $a$ and $b$, the idea will be to approximate with Agler subshifts and take a weak-* limit of Berger measures, using weak-* compactness. \  Some care must be taken that we do not collapse to the trivial completion above, and for this we require a lemma.

\begin{lem}  Suppose $A_m$ is one of the Agler shifts and we consider the weights $\sqrt{p} = \sqrt{\frac{n+1}{m+n}}$ and $\sqrt{q} = \sqrt{\frac{n +\Delta + 1}{m+n + \Delta}}$, as the first two weights of a potential sub-shift, where $\Delta$ is the spacing of the sub-shift and $\Delta \geq 1$. \  Observe that the third weight of the sub-shift is $\sqrt{r} = \sqrt{\frac{n +2\Delta + 1}{m+n + 2\Delta}}$. \  Let $G_1 = q-p$ (the ``gap'' between the first two weights) and $G_2 = r - q$, the second gap. \  Then
\begin{equation}      \label{eq:GapFormula}
\frac{G_2}{G_1} = \frac{1-q}{1-2p+q}.
\end{equation}
In particular, this fraction is bounded away from zero for any fixed $p$ and $q$ in $(0,1)$. \  Further, the expression on the right hand side of \eqref{eq:GapFormula} is increasing in $p$ and decreasing in $q$.
\end{lem}

\begin{proof}
Supposing the expression in \eqref{eq:GapFormula} to be correct, the increasing and decreasing claims are merely a matter of derivatives. \  The rest is mostly calculation:  first,
$$G_2 = \frac{\Delta  (m-1)}{(\Delta +m+n) (2 \Delta +m+n)}$$
and
$$G_1 = \frac{\Delta  (m-1)}{(m+n)(\Delta +m+n)}.$$
It follows that
$$\frac{G_2}{G_1} = \frac{m+n}{2 \Delta +m+n}.$$
One shows that to have the right weights requires
$$n = \frac{1-p m}{p-1}$$
and
$$\Delta = \frac{-q m -q n+n+1}{q-1}.$$
Inserting these into the expression for $\frac{G_2}{G_1}$ yields the claim.  
\end{proof}

\medskip

\begin{prop} \label{prop322}
Given values $a$ and $b$ satisfying $0 < a < b < 1$ it is possible to extend $\sqrt{a}, \sqrt{b}$ to the weight sequence of an $\mathcal{MID}$ contractive weighted shift which is not flat.
\end{prop}

\begin{proof}
Observe that it is easy to show that the weights squared of the $m$-th Agler shift (together with $1$) form a partition of $[\frac{1}{m}, 1]$ whose largest width is $\frac{m-1}{m (m+1)}$, which clearly goes to zero as $m$ becomes large. \  Thus, given $0 < a < b < 1$, we may clearly approximate $a$ and $b$ as closely as we like with weights squared of one of the Agler shifts. \  Further, we may take care to ensure that we approximate $a$ using a weight squared which is smaller, and to approximate $b$ using a weight squared which is larger. \ We will choose rational numbers $p_n$ and $q_n$ such that $a = \lim p_n$ and $b = \lim q_n$, and obtain as above the sub-shifts for the pairs $\sqrt{p_n}$ and $\sqrt{q_n}$. \  Each comes with a Berger measure, say $\mu_n$. \  Because of weak-* compactness, we may choose a weak-* convergent subsequence of the $\mu_n$, converging to some measure $\mu$. \ It is easy to check that $\mu$ is a probability measure and yields the correct first two weights $\sqrt{a}$ and $\sqrt{b}$.

Consider now the third weight of one of the approximating subshifts. \  It will have the ratio $\frac{G_2}{G_1}$ no smaller than $\frac{1-b}{1-2a+b}$, and hence these ratios for all the approximating Agler sub-shifts will be bounded away from zero. \  Since weight information is moment information, and hence integrals against $t$, $t^2$, and $t^3$ of the appropriate measure, this will be preserved in the measure limit. \  Thus the measure $\mu$ will yield a weighted shift with the gap ratio $\frac{G_2}{G_1}$ non-zero, and hence the third weight will not be equal to the second, and the resulting shift cannot be flat. \ It is evident that the procedure yields a contraction, and thus it will be $\mathcal{MID}$ by Proposition \ref{prop:wstarlimsofIDBerger}. \   This gives the desired result.  
\end{proof}

We pose some natural questions.
\begin{question}  \label{question:3wtscomp}
What condition on three initial weights guarantees an $\mathcal{MID}$ completion?  What condition is necessary and sufficient for such a completion?
\end{question}

\medskip

There is some information about the first of these questions. \ Recall that an operator $T$ is completely hyperexpansive if, for all $n = 1, 2, \ldots$, $\sum_{i=0}^n (-1)^i \binom{n}{i} {T^*}^i T^i \leq 0$. \  Since if $W_\alpha$ is completely hyperexpansive the shift with weight sequence the reciprocals of the $\alpha_j$ is $\mathcal{MID}$ (\cite[Corollary 4.1]{BCE}), we may use conditions from \cite{JJKS} to obtain a variety of sufficient conditions for completions. \  We content ourselves with the following, which is just reciprocals inserted into \cite[Proposition 5.2]{JJKS}.

\begin{prop}  \label{prop:3wtsviaCHE}
Consider three initial weights $0< \alpha_0 < \alpha_1 < \alpha_2 < 1$. \  If
\begin{description}
  \item[i)] $\frac{1}{\alpha_1^2} \frac{1}{\alpha_2^2} - 2 \frac{1}{\alpha_1^2} + 1 \leq 0$ and
  \item[ii)] $\alpha_2^2(1-\alpha_1^2)^2 \le (1-\alpha_0^2)(1-\alpha_2^2)\alpha_1^2$, 
\end{description}
then $\alpha_0, \alpha_1, \alpha_2$ has an extension to the weight sequence of an $\mathcal{MID}$ weighted shift.
\end{prop}

Since it is known that not every $\mathcal{MID}$ shift has a weight sequence whose reciprocals yield a completely hyperexpansive shift (see the discussion after \cite[Proposition 6]{At}), it is unsurprising that the condition above is not necessary.

\begin{ex}
Consider the proposed initial weight sequence $\sqrt{\frac{1}{3}}, \sqrt{\frac{2}{4}},\sqrt{\frac{3}{5}}$ (the first three weights of the third Agler shift). \  These weights fail condition (i) in Proposition \ref{prop:3wtsviaCHE}, yet clearly have a completion to an $\mathcal{MID}$ weight sequence. 
\end{ex}

We do not know whether it is possible to have a three weight initial segment which can be completed  in three ways:  (i) to subnormal but not $\mathcal{MID}$, (ii) to $\mathcal{MID}$ but not the reciprocal of a completely hyperexpansive shift, and (iii) to such a reciprocal.

Another sort of rigidity concerns an analog to a feature of the standard routes to subnormality of a shift, namely $k$-hyponormality and $n$-contractivity. \  Various sorts of ``extremality'' (in the sense that one of the standard positivity tests is actually achieved at zero) yield that the shift must be recursively generated (have finitely atomic Berger measure) or sometimes something even more restrictive. \  Perhaps the oldest result is that a subnormal weighted shift with two successive weights equal must be ``flat'' (all or all but one weight equal);  this is \cite[Theorem 6]{StJ}. \  This was later improved to shifts only assumed to be $2$-hyponormal in \cite[Corollary 6]{Cu}, and further to shifts assumed only to be quadratically hyponormal in \cite[Theorem 1]{Ch}. \  In \cite[Theorem 5.12, Proposition 5.13]{CF1} is the result that if, for some $k$,  $\det A(j,k) =0$  for all $j \geq 0$ then the shift must be recursively generated. \  In \cite[Theorem 2.3]{EJP} we see that if a contractive shift $W$ satisfies $A_n(W) = 0$ for some $n$, then $W$ must be the unweighted unilateral shift. \ (Recall the definition of $A_n(T)$, given in (\ref{eq11}).)

The following result is a moment infinite divisibility version of this same flavor.

\begin{prop}  \label{prop:LTeq0Flat} Suppose $W_\alpha$ is a contractive weighted shift which is $\mathcal{MID}$ (so its moment sequence is log completely monotone and its weight sequence is log completely alternating). \  Suppose further that, for some $n$ and $k$, $LT_\alpha(n,k) = 0$. \  Then $W_\alpha$ is flat.
\end{prop}

\begin{proof}
Suppose $n$ and $k$ are as in the statement, and consider first the case $n = 1$;  we have that $\log \gamma_k - \log \gamma_{k+1} = 0$. \  But this is $ - \log \alpha_k^2 = 0$, and so $\alpha_k = 1$, and since the weights are increasing and bounded above by $1$ we clearly have flatness. \  In the case that $n=2$, we have $\log \gamma_k - 2 \log \gamma_{k+1} + \log \gamma_{k+2} = 0$, which is easily equivalent to $\alpha_{k+1} = \alpha_k$ and again we have flatness.

Suppose now that $n=3$. \  We have
$$\log \gamma_k - 3 \log \gamma_{k+1} + 3 \log \gamma_{k+2} - \log \gamma_{k+3} = 0,$$
which is equivalent to
$$\frac{\alpha_k^2}{\alpha_{k+1}^2} = \frac{\alpha_{k+1}^2}{\alpha_{k+2}^2}.$$
Considering the shift $W_\beta$ with $\beta_k = \frac{\alpha_k}{\alpha_{k+1}}$, and since by Corollary \ref{cor:simplyweightquotient} $W_\beta$ is subnormal, we have that $W_\beta$ is flat. \  Therefore $\beta_{k} = \beta_{k+1} = \cdots = c$, for some $c$ satisfying $0 < c \leq 1$.

Therefore we have
$$c = \frac{\alpha_k^2}{\alpha_{k+1}^2} = \frac{\alpha_{k+1}^2}{\alpha_{k+2}^2} = \frac{\alpha_{k+2}^2}{\alpha_{k+3}^2} = \frac{\alpha_{k+3}^2}{\alpha_{k+4}^2} = \cdots .$$
Then  $\alpha_{k+2}^2 = \frac{1}{c} \alpha_{k+1}^2$. \  Then since $\frac{\alpha_{k+2}^2}{\alpha_{k+3}^2} = c$, we obtain that
$$\alpha_{k+3}^2 = \frac{1}{c}\alpha_{k+2}^2 = \frac{1}{c^2}\alpha_{k+1}^2.$$
Repeating the computation, we have that
$$\alpha_{k+j}^2 = \frac{1}{c^{j-1}}\alpha_{k+1}^2.$$
If $c < 1$, clearly the sequence $(\alpha_{k+j})_{j=1}^\infty$ is unbounded, a contradiction, so $c = 1$ and we have $\alpha_{k+2} = \alpha_{k+1}$ yielding flatness for $W_\alpha$, using again that $W_\alpha$ is subnormal. \  Observe that we have shown both the claim for $n=3$ and that if $W_\beta$ is flat, then $W_\alpha$ is flat in general.

The proof now finishes by induction, since if $n=4$ we deduce from $LT_\alpha(4,k) = 0$ that $LT_\beta(3,k) = 0$, that $W_\beta$ is flat by the result for $n=3$, hence $W_\alpha$ is flat, and so on.  
\end{proof}

Subnormality (instead of moment infinite divisibility) plus $LT_\alpha(n,k) = 0$ for some $n$ and $k$ is not sufficient to guarantee flatness, as shown by the following example.

\begin{ex}
Consider the weighted shift $W_{(a,b,c)}$ which is the Stampfli subnormal completion of the initial weights $a<b<c$, where $b:=\sqrt{ac}$. \ Using (\ref{eq1}), one easily checks that $LT_{\gamma}(3,0) = -LT_{\alpha^2}(2,0) = 0$. \ Moreover, $\left\|W_{(a,b,c)}\right\| \le \sqrt{c(a+c)}$. \ However, the shift is not flat. 
\end{ex}

We now record another collection of moment infinitely divisible weighted shifts (in fact, ones whose weights squared are completely alternating instead of merely log completely alternating). \  It is well known that the Agler shifts are simply representatives of a more general class of shifts with Berger measures $d \mu(t) = (j-1) (1-t)^{j-2} \chi_{[0,1]}(t) \, d t$, where $j > 1$;  the Agler shifts are simply those for which $j$ is an integer. \  Unsurprisingly, the others in the class are moment infinitely divisible.

\begin{prop}
Let $s > 1$ (not necessarily an integer) and let $W$ be the weighted shift with Berger measure $d \mu(t) = (s-1) (1-t)^{s-2} \chi_{[0,1]} (t)\, d t$. \  The weights squared $(\alpha^{(s)}_{n})^2$ of $W$ are completely alternating and therefore $W$ is $\mathcal{MID}$.
\end{prop}

\begin{proof}
One computes readily that the weights squared are $(\alpha^{(s)}_{n})^2 = \frac{1+n}{s+n}$ and that 
$(\nabla^{m} (\alpha^{(s)})^2)_n = \frac{(1-s) m!}{\prod_{i=0}^{m}(s+n+i)}$, and these are clearly negative.
\end{proof}

\medskip

\noindent \textbf{Acknowledgments.} \ The authors wish to express their appreciation for support and warm hospitality during various visits, which materially aided this work, to Bucknell University, the University of Iowa, and the Universit\'{e} des Sciences et Technologies de Lille, and particularly the Mathematics Departments of these institutions. \ Several examples in this paper were obtained using calculations with the software tool \textit{Mathematica} \cite{Wol}.

\end{document}